\documentclass[12pt]{amsart}
\usepackage{amscd,amssymb,hyperref}
\usepackage[PostScript=dvips]{diagrams}
\usepackage{amsfonts}
\usepackage[all]{xy}

\newtheorem{thm}{Theorem}[section]

\newtheorem{cor}[thm]{Corollary}
\newtheorem{prop}[thm]{Proposition}

\newtheorem{rem}[thm]{Remark}

\def\square{\vbox{
      \hrule height 0.4pt
      \hbox{\vrule width 0.4pt height 5.5pt \kern 5.5pt \vrule width 0.4pt}
      \hrule height 0.4pt}}

\def\id{\mathrm{id}}

\def\ch\mathrm{c h}
\def\ab{\mathrm{a b}}

\def\End{\mathrm{End}}

\def\End{\mathrm{End}}

\def\Aut{\mathrm{Aut}}

\def\twist{\mathrm{twist}}
\long\def\symbolfootnote[#1]#2{\begingroup%
\def\thefootnote{\fnsymbol{footnote}}\footnote[#1]{#2}\endgroup}

\newcommand{\Z}{\mathbb{Z}}

\newcommand{\calC}{\ensuremath{\mathcal{C}}}

\newcommand{\calQ}{\ensuremath{\mathcal{Q}}}
\newcommand{\calG}{\ensuremath{\mathcal{G}}}

\numberwithin{equation}{section}

\newcommand{\auths}[1]{\textrm{#1},}
\newcommand{\artTitle}[1]{\textsl{#1},}
\newcommand{\jTitle}[1]{\textrm{#1}}
\newcommand{\Vol}[1]{\textbf{#1}}
\newcommand{\Year}[1]{\textrm{(#1)}}
\newcommand{\Pages}[1]{\textrm{#1}}

\title[Twisted Simplicial Groups]{Twisted Simplicial Groups and Twisted Homology of Categories}

\author{J. Y. Li }
\address{Department of Mathematics and Physics, Shijiazhuang Tiedao University 050000, China}
\email{yanjinglee@163.com}

\author{V. V. Vershinin }
\address{D\'epartement des Sciences Math\'ematiques,
                                     Universit\'e de Montpellier,
Place Eug\`ene Bataillon,
34095 Montpellier cedex 5, France}
\email{vladimir.verchinine@univ-montp2.fr}
\address{Sobolev Institute of Mathematics, Novosibirsk 630090,
Russia }
\email{ versh@math.nsc.ru}
\address{Laboratory of Quantum Topology, Chelyabinsk State University, Brat'ev
Kashirinykh street 129, Chelyabinsk 454001, Russia}
\author{J. Wu }
\address{Department of Mathematics, National University of Singapore, 2 Science Drive 2
Singapore 117542} \email{matwuj@nus.edu.sg}
\urladdr{www.math.nus.edu.sg/\~{}matwujie}

\thanks{
The first author is supported by NSFC (11201314, 11302136) of China and
the Excellent Young Scientist Fund of Shijiazhuang Tiedao University. The second author is partially supported by the Laboratory of
Quantum Topology of Chelyabinsk State University (Russian Federation government grant 14.Z50.31.0020) and RFBR grants 14-01-00014
and 13-01-92697-IND. The last author is partially supported by the Singapore Ministry of Education research grant (AcRF Tier 1 WBS No.
R-146-000-190-112) and a grant (No. 11329101) of NSFC of China.}

\subjclass[2000]{Primary 55U10; Secondary 18G30}
\begin{document}

\begin{abstract}
Let $A$ be either a simplicial complex $K$ or a small category $\calC$ with $V(A)$
as its set of vertices or objects.
%For a simplicial group $G$
We define a
twisted
    structure on $A$ with
  coefficients in a simplicial group $G$ as a function
$$
\delta\colon V(A)\longrightarrow \End(G), \quad v\mapsto \delta_v
$$
such that
%satisfying the following conditions:
%\begin{equation*}\label{equation_a.1}
$
\delta_v\circ \delta_w=\delta_w\circ \delta_v
$
%\end{equation*}
if there exists an edge in $A$  joining $v$ with $w$ or
 an arrow either from $v$ to $w$ or from $w$ to $v$.
 We give a canonical
 construction of twisted simplicial group as well as twisted
  homology for $A$
  %a simplicial complex  a  small category
  with a given twisted structure.
Also we determine the homotopy type of of this simplicial group
 as the loop space over certain  twisted smash product.
\end{abstract}

\maketitle

%\tableofcontents
\section {Introduction}

Motivated by the applications of algebraic topology to dynamic processes in the network and other areas of mathematics, some new versions of homology (cohomology) theory and combinatorial homotopy theory of graphs and simplexes were introduced recently~\cite{BL,BKLW,CM,Everitt-Turner,GMY,GMY2, GLMY,TB, Turner-Wagner}. In particular, the classical simplicial homology was varied by considering homology with local coefficient system for colored posets~\cite{Everitt-Turner, Turner-Wagner} or by inserting $\delta$-factors as numerical data on vertices in the differentials of simplicial chain complexes~\cite{GMY,GMY2}. The purpose of this paper is to give a canonical twisted construction of simplicial groups as well as twisted homology for a simplicial complex and a  small category with a given twisted structure. In particular, our twisted homology (twisted cohomology) is a generalization of the $\delta$-homology ($\delta$-cohomology) introduced in~\cite{GMY, GMY2}. Categories con!
 sidered in the paper are small and we omit  mentioning it. All necessary definitions will be given later in the text of the paper.

Let $A$ be either a simplicial complex $K$ or a category $\calC$. We are considering
 category from  geometrical
 point of view, using the terminology of underlying quiver, for example we are
 using word "vertices" instead of objects. However  composition of morphisms and identity morphisms are involved
 in our constructions.
 %remind the definition at the beginning
% of subsection~\ref{subsection2.1})
Let $V(A)$ be the vertex set of $A$, $G$ be a simplicial group. We denote by
$\End(G)$  the monoid of simplicial endomorphisms of $G$, namely $\End(G)$ consists
   of graded endomorphisms $f=\{f_n\colon G_n\to G_n\}_{n\geq0}$ such that
   $d_i\circ f=f\circ d_i$ and $s_i\circ f=f\circ s_i$
   for the face homomorphisms $d_i\colon G_n\to G_{n-1}$ and the degeneracy
   homomorphisms $s_i\colon G_n\to G_{n+1}$ for $0\leq i\leq n$.  A \textit{twisted
    structure} on $A$ with
  coefficients in a simplicial group $G$ is a function
$$
\delta\colon V(A)\longrightarrow \End(G), \quad v\mapsto \delta_v
$$
satisfying the following commuting rule:
\begin{equation}\label{equation1.1}
\delta_v\circ \delta_w=\delta_w\circ \delta_v
\end{equation}
if there exists an edge in $A$ (if $A$ is a simplicial complex) joining $v$ with $w$ or
 an arrow either from $v$ to $w$ or from $w$ to $v$ (if $A$ is a category).
 % where
 % $\End(G)$ is the monoid of simplicial endomorphisms of $G$, namely $\End(G)$ consists
 %  of graded endomorphisms $f=\{f_n\colon G_n\to G_n\}_{n\geq0}$ such that
 %  $d_i\circ f=f\circ d_i$ and $s_i\circ f=f\circ s_i$
 %  for the face homomorphisms $d_i\colon G_n\to G_{n-1}$ and the degeneracy
  % homomorphisms $s_i\colon G_n\to G_{n+1}$ for $0\leq i\leq n$.
%In the case $G$ is a discrete simplicial group (that is $G_n=G_0$ for $n\geq0$ and each
% face or degeneracy is the identity), the twisted structure on $K$ with coefficients in
 %$G$ is given by a representation of the vertices of $K$ in $G$ with the above commuting
  %rule for the vertices joining by edges.
A twisted structure on $A$ with coefficients in $G$  is called \textit{non-singular} if $\delta_v\colon G\to G$ is a simplicial automorphism for each $v\in V(A)$.

Now let us suppose that a simplicial complex  $K$ has a total order.
There is a canonical associated
  simplicial set $S(K)$
  %(or $S(\calC)$)
  by allowing the repeating of vertices in the
   sequences of paths.
%the regular path complex \cite[Definition~3.4]{GLMY2} of
%a  category $\calC $ with a twisted structure
% $\delta$ with coefficients in a simplicial group $G$.
   %(See section~\ref{section2} for details.)
   For a category $\calC$ take $S=S(\calC)$ to be a nerve of $\calC$.
   The simplicial set $S=S(A)$ %(or $S(K)$)
   has a nature that the set $S_n$  of $n$-paths ($n$-simplices) admits a form of the sequences
$$v_0\alpha_1v_1\cdots\alpha_nv_n$$
with $v_i\in V(\calC)$ and
$\alpha_i\colon v_{i-1}\to v_i$ is a morphism or the sequences
 $(v_0, v_1,\ldots,v_n)$ for
$v_i\in V(K)$ with $v_i\leq v_j$ for $i\leq j$ under a total order on $V(K)$ and $\{v_0,\ldots,v_n\}$ being a simplex in $K$. The face operation
$d_i\colon S_n\to S_{n-1}$, $0\leq i\leq n$, is given by removing the vetex $v_i$ followed the composition $\alpha_i\alpha_{i+1}\colon v_{i-1}\to v_{i+1}$, and the
 degeneracy operation $s_i\colon S_n\to S_{n+1}$, $0\leq i\leq n$, is given by doubling
  the vertex $v_i$ with inserting the identity arrow $e_{v_i}\colon v_i\to v_i$ if it is
   the case of category.

  We
    construct %in section~\ref{section3}
    a $\Delta$-group and a simplicial group
     $F^G_{\delta}[A]$
    depending on the twisted structure $\delta$.
This $\Delta$-group  seems to be interesting. If the $\delta$-structure is given by
 $\delta_v\equiv \id_{G}$ for $v\in V(A)$,
%or $V(\calC)$,
then $F^G_{\delta}[A]$ is Carlsson's $J_G(X)$-construction~\cite{Carlsson} or Quillen's
 tensor product~\cite{Quillen2} with its geometric realization having the homotopy type
  of $\Omega(B|G|\wedge |A|)$, where $|X|$ is the geometric realization of a
  simplicial set $X$.

The main result in this article is to determine the homotopy type of  $F^G_{\delta}[A]$
 as the loop space of the \textit{twisted smash product} $B|G|\wedge_{\delta} |A|$.
 See section~\ref{section4} for the definition of twisted smash product.

%We should point out that the twisted $\Delta$-groups and simplicial groups constructed
% in this article are only valid for categories.

%The twisted idea introduced in this article may be applied for twisted homology or
% homotopy of more general categorys, however some technical difficulties need to be
% overcome for establishing homology theory or homotopy theory for general categorys due
% to the fact that the composition operation on the arrows may not be well-defined in
% general.

The article is organized as follows. In Section~\ref{section2}, we give a review on the
 path complexes of categories. The twisted construction of the $\Delta$-groups and
  simplicial groups
  %described in Steps 2 and 3
   is given in section~\ref{section3} by examining the $\Delta$-identity and simplicial identities. The homotopy type of the twisted simplicial group $F^G_{\delta}[A]$ is studied in section~\ref{section4}, where the main result is Theorem~\ref{main-theorem}. In section~\ref{section5}, we explore the twisted homology of simplicial complexes as well as   categories.

\section{Path complexes of categories and simplicial set extension of simplicial complexes}\label{section2}
\subsection{Path complexes of categories}\label{subsection2.1}
Recall that a quiver is a quadruple $\mathcal{Q}=(V(\calQ),E(\calQ), s, t)$,
where the set $V$ is called the set of vertices of $\mathcal{Q}$,
    the set $E$ is called the set of edges of $\mathcal{Q}$,
    $s$ and $t$ are two maps
    $ s: E \to V$, giving the start or source of the edge, and $t: E \to V$,
    giving the target of the edge.
The
elements in $E(\calQ)$ we denote by \textit{arrows} $\alpha\colon v\to w$, or
$v\alpha w$.

Recall also that %a \textit{
there is a forgetful functor from small categories to quivers
$$
{\bf\operatorname{\bf Cat}}\to {\bf\operatorname{\bf Quiv}}.
$$
%For simplicity we consider a quiver as a pair $\mathcal{Q}=(V(\calC),E(\calC))$, where
% $V(\calC)$ is a set of vertices and $A(\calC)$ is a set of  arrows between elements of
% $V(\calC)$. The elements in $V(\calC)$ are called \textit{vertices}.

%In other words there exists not more that one arrow for an ordered couple $(v,w)$. Note that this is a special case of a quiver, where several arrows are possible, in particular several loops.

By  an $n$-path for a category $\calC$ we mean the sequence
$v_0\alpha_1\cdots  v_{n-1}\alpha_{n}v_{n}$ with $n\geq1$. There exists a trivial path
 (the identity arrow) $e_v\colon v\to v$ for each vertex $a\in V(\calC)$.
Let $S_n(\calC)$ be the set of all $n$-paths with $n\geq1$  and $S_0(\calC)=V(\calC)$.

%We introduce With having the trivial paths $e_v\colon v\to v$, we can
The face operations
% in $\calK(\calC)$ (and so
 in $S(\calC)$ are given by removing a vertex followed by
 the composition of arrows in the following sense:
\begin{equation}\label{equation2.1}
d_i(v_0\alpha_1v_1\alpha_2v_2\cdots \alpha_nv_n)=v_0\alpha_1v_1\cdots v_{i-1}
(\alpha_{i+1}\circ \alpha_{i})v_{i+1}\cdots\alpha_nv_n
\end{equation}
for $0\leq i\leq n$, where $d_0$ removes the first vertex together with
the first arrow and
 the last face $d_n$ removes the last vertex together with the last arrow.
%We introduce t
The \textit{degeneracy operations}
$$
s_i\colon S_n(\calC)\longrightarrow S_{n+1}(\calC)
$$
for $0\leq i\leq n$ are defined  by doubling the $i$-vertex with inserting the
trivial path.
Then $S(\calC)=\{S_n(\calC)\}_{n\geq0}$ is the classifying
 simplicial set (or \textit{nerve}) of the category $\calC$.
Choose a vertex $a_0\in V(\calC)$ as a basepoint. Consider the constant $n$-path $a_0^{n}=a_0e_{a_0}a_0\cdots e_{a_0}a_0$ as a basepoint for the set $S_n(\calC)$. Then $S(\calC)$ is a pointed simplicial set \cite[Chapter~2]{Wu}.

\subsection{Simplicial set extension of simplicial complexes}

Let $K$ be a simplicial complex with its vertex set $V(K)$ which has a total order. %on its vertices.
Recall that the simplicial set $S(K)$ induced by $K$ is given as follows:
\begin{enumerate}
\item The set $S(K)_n$ consists of sequences $(v_0,v_1,\ldots,v_n)$ of vertices of $K$ such that
$v_0\leq v_1\leq\cdots \leq v_n$ and  $\{v_0,\ldots,v_n\}$ forms a simplex in  $K$.
\item The face operation $d_i\colon S(K)_n\to S(K)_{n-1}$,  $0\leq i\leq n$, is given by
\begin{equation}\label{equation1}
d_i(v_0,v_1,\ldots,v_n)=(v_0,v_1,\ldots,v_{i-1},v_{i+1},\ldots, v_n).
\end{equation}
\item The degeneracy operation $s_i\colon S(K)_n\to S(K)_{n+1}$, $0\leq i\leq n$, is given by
\begin{equation}\label{equation2}
s_i(v_0,v_1,\ldots,v_n)=(v_0,v_1,\ldots,v_{i-1},v_i,v_i,v_{i+1},\ldots, v_n).
\end{equation}
\end{enumerate}
The geometric realization of the simplicial set $S(K)$ is homeomorphic to the polyhedron $|K|$ (see, for example, \cite{Curtis}). We choose a basepoint in $S(K)_n$  by $a^n_0=(a_0,a_0,\ldots,a_0)$ for a vertex $a_0\in V(K)$ so that $S(K)$ is a pointed simplicial set.

\section{The twisted construction of $\Delta$-groups and simplicial Groups}\label{section3}

A $\Delta$-{\it set} is a sequence of sets $X = \{{X_n}, {n\geq 0}\}$ with maps  called
{\it faces}
$$
d_i : X_n \to X_{n-1}, \ 0 \leq i \leq n,$$
 such that the following condition is hold
\begin{equation} d_id_j = d_jd_{i+1}
\label{Delta-eq}
\end{equation}
for $i \geq j$. It is called the $\Delta$-{\it identity} \cite{Wu}.
Roughly speaking $\Delta$-set is a simplicial set without degeneracies.

A $\Delta$-set $G = \{G_n\}_{n\geq 0}$ is called a $\Delta$-{\it group} if each $G_n$
is a group, and each face
$d_i$ is a group homomorphism \cite{Wu}.

 We
   proceed the following steps for constructing a $\Delta$-group and a simplicial group
    depending on the twisted structure $\delta$ on $A$:
\begin{enumerate}
\item[1)] Choose a vertex $a_0\in V(A)$
%(or $V(K)$)
as a basepoint. Consider the constant $n$-path $a_0^{n}=a_0e_{a_0}a_0\cdots e_{a_0}a_0$
 as a basepoint for the set $S_n$. Define $F^G[A]_n$ to be the free product of the group
  $G_n$ with indices running over all non-basepoint elements in $S_n(A)$. More precisely,
   let $(G_n)_x$ be a copy of $G_n$ labelled by an element $x\in S_n(A)$. Then
   $F^G[A]_n$ is the quotient group of the free product
$
\ast_{x\in S_n} (G_n)_x
$
subject to the relations that $(G_n)_{a_0^{n}}=\{1\}$.
\item[2)] Let $g_x$ denote the element $g\in G_n$ in its copy group $(G_n)_x$ labelled
 by $x$. Define a (twisted) face operation $d_i^{\delta}\colon F^G[A]_n\to F^G[A]_{n-1}$,
  $0\leq i\leq n$ to be the (unique) group homomorphism such that $d^{\delta}_i$
   restricted to each copy $(G_n)_{v_0\alpha_1v_1\cdots\alpha_nv_n}$ is given by the
  twisted formula
\begin{equation}\label{twisted-face}
d^{\delta}_i|_{(G_n)_{v_0\alpha_1v_1\cdots\alpha_nv_n}}(g_{v_0\alpha_1v_1\cdots\alpha_nv_n})=(\delta_{v_i}(d_i(g)))_{d_i(v_0\alpha_1v_1\cdots\alpha_nv_n)}.
\end{equation}
Namely the element $d_ig\in (G_{n-1})_{d_i(v_0\alpha_1v_1\cdots\alpha_nv_n)}$ is twisted by the endomorphism $\delta_{v_i}\colon G_{n-1}\to G_{n-1}$
By Proposition~\ref{monoid} below, the sequence of groups
$F^{G,\Delta}_{\delta}[A]=\{F^G[A]_n\}_{n\geq0}$ with the twisted face operations
forms a $\Delta$-group, where the commuting rule ~(\ref{equation1.1}) assures the
 $\Delta$-identity.
\item[3)] Suppose that the $\delta$-structure is non-singular. We define a twisted
 degeneracy operation $s_i^{\delta}\colon F^G[A]_n\to F^G[A]_{n+1}$, $0\leq i\leq n$, to
  be the (unique) group homomorphism such that $s^{\delta}_i$ restricted to each copy
   $(G_n)_{v_0\alpha_1v_1\cdots\alpha_nv_n}$ is given by the twisted formula
\begin{equation}\label{twisted-degeneracy}
s^{\delta}_i|_{(G_n)_{v_0\alpha_1v_1\cdots\alpha_nv_n}}(g_{v_0\alpha_1v_1\cdots\alpha_nv_n})=((\delta_{v_i})^{-1}(s_i(g)))_{s_i(v_0\alpha_1v_1\cdots\alpha_nv_n)}.
\end{equation}
By Proposition~\ref{monoid}, the sequence of groups
$F^{G}_{\delta}[A]=\{F^G[A]_n\}_{n\geq0}$ with the twisted face operations
$d^{\delta}_i$ in step~2 and the above twisted degeneracy operations $s_i^{\delta}$
forms a simplicial group \cite[Chapter 2]{Wu}, where the commuting rule ~(\ref{equation1.1}) and invertibility of $\delta_v$ are essential for assuring the simplicial identities.
\end{enumerate}

The $\Delta$-group $F^{G, \Delta}_{\delta}[A]$ is a generalization of
$\delta$-homology~\cite{GMY,GMY2} in the following sense. Let $G$ be any abelian group.
% Let $K$ be a simplicial complex or a  and transitive category with the vertex set
% $V(K)$.
Let $\delta\colon V(A)\to \End(G)$ be a function satisfying the commuting
 rule~(\ref{equation1.1}). We consider $G$ as a discrete simplicial group with $G_n=G$
  and the faces and degeneracies given by the identity. Then we have the
  $\delta$-structure on $A$. By taking the abelianization of the (non-commutative)
   $\Delta$-group $F^{G,\Delta}_{\delta}[A]$ in step~2, we have an abelian
   $\Delta$-group with the twisted faces given in~(\ref{twisted-face}) on chains
$$
(F^{G}[A]_n)^{\ab}=\left(\bigoplus_{x\in S_n}G_x\right)/(G)_{a_0^{n}}
$$
having the differentials given by
$$
\partial_n(g_{v_0\alpha_1v_1\cdots\alpha_nv_n})=\sum_{i=0}^n(-1)^i(\delta_{v_i
}(d_i(g)))_{d_i(v_0\alpha_1v_1\cdots\alpha_nv_n)}.
$$
In the case where $G$ is a commutative ring $R$ and $\delta_v$ is the translation given
by an element $\delta_v\in R$, the above formula on the differentials is exactly
given by inserting the factor $\delta_v$ as described in~\cite{GMY,GMY2}.

\begin{prop}\label{proposition3.1}
Let $S$ be $S(A)$ where $A$ is either  a simplicial complex $K$ or %$S(\calC)$ for
a category
$\calC$ with twisted structure in a simplicial group $G$. Then
\begin{enumerate}
\item The sequence of groups $F^{G,\Delta}_{\delta}[A]=\{F^G[A]_n\}_{n\geq0}$ with the
 twisted face operations defined by formula~\ref{twisted-face} forms a $\Delta$-group.
\item Suppose that the twisted structure $\delta$ is nonsingular. Then the sequence of
 groups $F^{G}_{\delta}[A]=\{F^G[A]_n\}_{n\geq0}$ with the twisted face operations
  $d^{\delta}_i$ given by formula~\ref{twisted-face} and the twisted degeneracy
   operations $s_i^{\delta}$ given by formula~\ref{twisted-degeneracy} forms a simplicial
    group.
\end{enumerate}
\label{monoid}
\end{prop}
\begin{proof}
The proof is given by examining the $\Delta$-identity for assertion (1), and simplicial
 identities for assertion (2). Since the examination on the $\Delta$-identity for
  assertion (1) is part of the simplicial identities for assertion (2), we only need to
   prove assertion (2) with paying attention that the invertibility of $\delta_v$ is not
    required in the $\Delta$-identity for face operations.

For the  basepoint $a^n_0$ in $S_n$, we see that each $d^{\delta}_i$ sends
 $(G_n)_{a^n_0}$ into $(G_{n-1})_{a^{n-1}_0}$, so
$$
d^{\delta}_i\colon F^{G}_{\delta}[A]_n\to F^{G}_{\delta}[A]_{n-1}
$$
is a well-defined homomorphism. The same is true for
$$
s^{\delta}_i\colon F^{G}_{\delta}[A]_n\to F^{G}_{\delta}[A]_{n+1}.
$$

Now we show that the simplicial-identity holds for the face and degeneracy operations.
 For simplicity of notations, we only consider the case $S=S(K)$. For the case
 $S(\calC)$ of a category $\calC$, the proof follows from the same lines with keeping
 in mind the following rules:
\begin{enumerate}
%\item[1)] The sequence $(v_0,\ldots,v_n)$ considered as an $n$-semi-path in
%$S_n(\calC)$ is multi-counted by having an arrow from $v_i\to v_{i+1}$.
\item[1)] When a vertex $v_i$ is removed from the sequence $(v_0,\ldots,v_n)$, the
arrow from $v_{i-1}\to v_{i+1}$ in the sequence $(v_0,\ldots,v_{i-1},v_{i+1},\ldots,v_n)$
 is given by the  composition $v_{i-1}\to v_i\to v_{i+1}$.
\item[2)]  When a vertex $v_i$ is doubled from the sequence $(v_0,\ldots,v_n)$, we
insert the trivial path $e_{v_i}\colon v_i\to v_i$ in the sequence
$$(v_0,\ldots,v_{i-1},v_i,v_i, v_{i+1},\ldots,v_n).$$
\end{enumerate}

Let $g_{(v_0,\ldots,v_n)}\in (G_n)_{(v_0,\ldots,v_n)}$, and let $i\geq j$ with
$0\leq j\leq n-1$ and $0\leq i\leq n$. Then
$$
\begin{array}{rcl}
d^{\delta}_i\circ d^{\delta}_j(g_{(v_0,\ldots,v_n)})&=&d^{\delta}_i(\delta_{v_j}^{n-1}(d_jg)_{(v_0,\ldots,v_{j-1},v_{j+1},\ldots,v_n)})\\
&=&d^{\delta}_i(d_j(\delta_{v_j}^{n}(g))_{(v_0,\ldots,v_{j-1},v_{j+1},\ldots,v_n)})\\
&=&(\delta^{n-2}_{v_{i+1}}(d_i(d_j(\delta_{v_j}^{n}(g)))))_{(v_0,\ldots,v_{j-1},v_{j+1},\ldots,v_i,v_{i+2},\ldots,v_n)}\\
&=&(d_i(d_j(\delta^{n}_{v_{i+1}}(\delta_{v_j}^{n}(g)))))_{(v_0,\ldots,v_{j-1},v_{j+1},\ldots,v_i,v_{i+2},\ldots,v_n)}\\
&=&(d_j(d_{i+1}(\delta^{n}_{v_{j}}(\delta_{v_{i+1}}^{n}(g)))))_{(v_0,\ldots,v_{j-1},v_{j+1},\ldots,v_i,v_{i+2},\ldots,v_n)}\\
&=&(d_j(\delta^{n-1}_{v_{j}}(d_{i+1}(\delta_{v_{i+1}}^{n}(g)))))_{d_jd_{i+1}(v_0,\ldots,v_n)}\\
&=&(\delta^{n-2}_{v_{j}}(d_j(\delta_{v_{i+1}}^{n-1}(d_{i+1}(g)))))_{d_jd_{i+1}(v_0,\ldots,v_n)}\\
&=&d^{\delta}_j\circ d^{\delta}_{i+1}(g_{(v_0,\ldots,v_n)}),
\end{array}
$$
where $\delta^n_{v_{i+1}}\circ \delta^n_{v_j}=\delta^n_{v_j}\circ \delta^n_{v_{i+1}}$ because $v_jv_{i+1}$ forms a $1$-simplex in $K$. (In the case of $S=S(\calC)$, $\delta^n_{v_{i+1}}\circ \delta^n_{v_j}=\delta^n_{v_j}\circ \delta^n_{v_{i+1}}$ because there is an arrow given by a composition from $v_j$ to $v_{i+1}$.)

Let $i\geq j$, and let $g_{(v_0,\ldots,v_n)}\in (G_n)_{(v_0,\ldots,v_n)}$. We shorten the notation $s^{\delta}_i$ as $s_i$. Then
$$
\begin{array}{rcl}
s_js_i(g_{(v_0,\ldots,v_n)})&=&s_j((\delta_{v_i}^{n+1})^{-1}(s_i(g))_{(v_0,\ldots,v_i,v_i,\ldots,v_n)})\\
&=&s_j(s_i((\delta_{v_i}^{n})^{-1}(g))_{(v_0,\ldots,v_i,v_i,\ldots,v_n)})\\
&=&(\delta_{v_j}^{n+2})^{-1}(s_js_i((\delta_{v_i}^{n})^{-1}(g)))_{(v_0,\ldots,v_j,v_j,\ldots,v_i,v_i,\ldots,v_n)}\\
&=&(s_js_i(\delta_{v_j}^{n})^{-1}((\delta_{v_i}^{n})^{-1}(g)))_{(v_0,\ldots,v_j,v_j,\ldots,v_i,v_i,\ldots,v_n)}\\
&=&(s_js_i(\delta_{v_i}^{n})^{-1}((\delta_{v_j}^{n})^{-1}(g)))_{(v_0,\ldots,v_j,v_j,\ldots,v_i,v_i,\ldots,v_n)}\\
&=&(\delta_{v_i}^{n+2})^{-1}(s_{i+1}((\delta_{v_j}^{n+1})^{-1} s_j(g)))_{(v_0,\ldots,v_j,v_j,\ldots,v_i,v_i,\ldots,v_n)}\\
&=&s_{i+1}s_j(g_{(v_0,\ldots,v_n)}).\\
\end{array}
$$
We shorten the notation $d^{\delta}_i$ as $d_i$. Let $i<j$, and let $g_{(v_0,\ldots,v_n)}\in (G_n)_{(v_0,\ldots,v_n)}$. Then
$$
\begin{array}{rcl}
d_is_j(g_{(v_0,\ldots,v_n)})&=&d_i((\delta_{v_j}^{n+1})^{-1}(s_j(g))_{(v_0,\ldots,v_j,v_j,\ldots,v_n)})\\
&=&d_i((s_j(\delta_{v_j}^{n})^{-1}(g)_{(v_0,\ldots,v_j,v_j,\ldots,v_n)})\\
&=&\delta_{v_i}^n(d_i(s_j(\delta_{v_j}^{n})^{-1}(g)))_{(v_0,\ldots,v_{i-1},v_{i+1},\ldots, v_j,v_j,\ldots,v_n)}\\
&=&(d_is_j\delta_{v_i}^n(\delta_{v_j}^{n})^{-1}(g))_{(v_0,\ldots,v_{i-1},v_{i+1},\ldots, v_j,v_j,\ldots,v_n)}\\
&=&(s_{j-1}d_i(\delta_{v_j}^n)^{-1}(\delta_{v_i}^{n})(g))_{(v_0,\ldots,v_{i-1},v_{i+1},\ldots, v_j,v_j,\ldots,v_n)}\\
&=&((\delta_{v_j}^n)^{-1}(s_{j-1}(\delta_{v_i}^{n-1}(d_i(g)))))_{(v_0,\ldots,v_{i-1},v_{i+1},\ldots, v_j,v_j,\ldots,v_n)}\\
&=&s_{j-1}d_i(g_{(v_0,\ldots,v_n)}).
\end{array}
$$
Let $i=j$, and let $g_{(v_0,\ldots,v_n)}\in (G_n)_{(v_0,\ldots,v_n)}$. Then
$$
\begin{array}{rcl}
d_js_j(g_{(v_0,\ldots,v_n)})&=&d_j((\delta_{v_j}^{n+1})^{-1}((s_j(g)))_{(v_0,\ldots,v_j,v_j,\ldots,v_n)})\\
&=&d_j((s_j(\delta_{v_j}^{n})^{-1}(g))_{(v_0,\ldots,v_j,v_j,\ldots,v_n)})\\
&=&(\delta_{v_j}^n(d_js_j(\delta_{v_j}^{n})^{-1}(g)))_{(v_0,\ldots,v_n)}\\
&=&(d_js_j(\delta_{v_j}^n(\delta_{v_j}^{n})^{-1}(g)))_{(v_0,\ldots,v_n)}\\
&=&g_{(v_0,\ldots,v_n)}.\\
\end{array}
$$
Let $i=j+1$, and let $g_{(v_0,\ldots,v_n)}\in (G_n)_{(v_0,\ldots,v_n)}$.
We remind that the rule is such that we add the twisting of the vertex we are removing.
When we apply $d_{j+1}$, we should add twisting $\delta_{v_j}$ because the second
$v_j$ is located in the position $(j+1)$ under removing.
Then we have
$$
\begin{array}{rcl}
d_{j+1}s_j(g_{(v_0,\ldots,v_n)})&=&d_{j+1}(\delta_{v_j}^{n+1})^{-1}(s_j(g)))_{(v_0,\ldots,v_j,v_j,\ldots,v_n)})\\
&=&d_{j+1}((s_j(\delta_{v_j}^{n})^{-1}(g)_{(v_0,\ldots,v_j,v_j,\ldots,v_n)})\\
&=&\delta_{v_j}^n(d_{j+1}s_j(\delta_{v_j}^{n})^{-1}(g))_{(v_0,\ldots,v_n)}\\
&=&(d_{j+1}s_j(\delta_{v_j}^n(\delta_{v_j}^{n})^{-1}(g)))_{(v_0,\ldots,v_n)}\\
&=&g_{(v_0,\ldots,v_n)}.\\
\end{array}
$$
Let $i>j+1$, and let $g_{(v_0,\ldots,v_n)}\in (G_n)_{(v_0,\ldots,v_n)}$.
Then
$$
\begin{array}{rcl}
d_is_j(g_{(v_0,\ldots,v_n)})&=&d_i((\delta_{v_j}^{n+1})^{-1}(s_j(g))_{(v_0,\ldots,v_j,v_j,\ldots,v_n)})\\
&=&d_i((s_j(\delta_{v_j}^{n})^{-1}(g)_{(v_0,\ldots,v_j,v_j,\ldots,v_n)})\\
&=&\delta_{v_{i-1}}^n(d_i(s_j(\delta_{v_j}^{n})^{-1}(g)))_{(v_0,\ldots,v_j,v_j,\ldots,v_{i-2},v_{i},\ldots, v_n)}\\
&=&(d_is_j\delta_{v_{i-1}}^n(\delta_{v_j}^{n})^{-1}(g))_{(v_0,\ldots, v_j,v_j,\ldots,v_{i-2},v_{i},\ldots,v_n)}\\
&=&(s_{j}d_{i-1}(\delta_{v_j}^n)^{-1}(\delta_{v_{i-1}}^{n})(g))_{(v_0,\ldots,v_j,v_j,\ldots,v_{i-2},v_{i},\ldots, v_n)}\\
&=&((\delta_{v_j}^n)^{-1}(s_{j}(\delta_{v_{i-1}}^{n-1}(d_{i-1}(g)))))_{(v_0,\ldots, v_j,v_j,\ldots,v_{i-2},v_{i},\ldots,v_n)}\\
&=&s_{j}d_{i-1}(g_{(v_0,\ldots,v_n)}).
\end{array}
$$
\end{proof}

\section{Homotopy Type of $F^{\mathcal{G}}_{\delta}[S]$}\label{section4}

\subsection{The $\delta$-twisted Cartesian  products}\label{subsection4.1}

We generalize now the notion of twisted structure for a simplicial set, not
necessary
a simplicial group $G$. Let $Y$ be a simplicial set. We denote by $\End(Y)$  the monoid of self-simplicial maps of $Y$.
Let $A$ be a simplicial complex or a category with vertex set $V(A)$.  A \textit{twisted structure} on $A$ in a simplicial set $Y$ is a function
$$
\delta\colon V(A)\longrightarrow \End(Y), \quad v\mapsto \delta_v
$$
such that the commuting rule
\begin{equation}\label{commuting-rule2}
\delta_v\circ\delta_w=\delta_w\circ\delta_v
\end{equation}
holds if there is an edge joint $v$ and $w$ in the case of simplicial complexes and arrow between $v$ and $w$ in the case of category. A twisted structure on $A$ is called \textit{nonsingular} if $\delta_v\colon Y\to Y$ is a simplicial isomorphism for each $v\in V(A)$.

Let $S=S(A)$ be a simplicial set defined in section~\ref{section2}.
% $K$ or a  and transitive category $\calC$.
Given any twisted structure $\delta$ on $A$  in a simplicial set $Y$, we can proceed the following construction of \textit{$\delta$-twisted Cartesian product}:
\begin{enumerate}
\item The sequence of sets is given by $Y_n\times S_n$ for $n\geq 0$.
\item The faces $d_i\colon Y_n\times S_n \to Y_{n-1}\times S_{n-1}$, $0\leq i\leq n$,
are defined by the formula
\begin{equation}\label{face-rule}
d_i(y, v_0\alpha_1v_1\cdots\alpha_nv_n)=(\delta_{v_i}(d_iy),
 d_i(v_0\alpha_1v_1\cdots\alpha_nv_n).
\end{equation}
\item Suppose that the twisted structure $\delta$ is nonsingular. We define the
 degeneracy operation $s_i\colon Y_n\times S_n\to Y_{n+1}\times S_{n+1}$, $0\leq i\leq
  n$, by the formula
\begin{equation}\label{degeneracy-rule}
s_i(y, v_0\alpha_1v_1\cdots\alpha_nv_n)=(\delta_{v_i}^{-1}(s_iy),
 s_i(v_0\alpha_1v_1\cdots\alpha_nv_n).
\end{equation}
\end{enumerate}

\begin{rem}
The definition of our $\delta$-twisted Cartesian product is different from that the classical twisted Cartesian product in simplicial theory~\cite{BGM, Curtis}.
\end{rem}

\begin{prop}\label{twisted-Cartesian}
Let $S=S(A)$ for $A$
%(or $S(\calC)$) for a simplicial complex $K$ (or a  and transitive category $\calC$)
 with a twisted structure $\delta$ in a simplicial set $Y$. Then
\begin{enumerate}
\item The sequence of sets $Y\times_{\delta,\Delta}S=\{Y_n\times S_n\}_{n\geq0}$ with the face operations defined by~(\ref{face-rule}) forms a $\Delta$-set.
\item Suppose that the twisted structure $\delta$ is nonsingular. Then the sequence of sets $Y\times_{\delta}S=\{Y_n\times S_n\}_{n\geq 0}$ with the face operations defined in~(\ref{face-rule}) and the degeneracy operations defined by~(\ref{degeneracy-rule}) forms a simplicial set.
\end{enumerate}
\end{prop}
\begin{proof}
The assertions follow from the lines in the proof of Proposition~\ref{proposition3.1} for examining the $\Delta$-identity and the simplicial identities.
\end{proof}

The first assertion of the above proposition is sufficient for having homology
 (cohomology) with twisted coefficients. Let $A$
 %be a simplicial complex or a  and transitive category with
have a twisted structure $\delta$ in a simplicial set $Y$. Let $G$ be an abelian group.
 The \textit{twisted homology} (\textit{twisted cohomology}) of $A$ with coefficients
 in $G$ is defined by
$$
H_n^{\twist}(A;G)=H_n(G(Y\times_{\delta,\Delta}S(A))), \ H^n_{\twist}(A;G)=H^n(G(Y\times_{\delta,\Delta}S(A)))
$$
for $n\geq0$, where $G(X)$ for a simplicial set $X$ is given by
$G(X)=\oplus_{x\in X}(G)_x$ with $(G)_x$ a copy of $G$ labelled by $x$.
A consequence of (1) of Proposition~(\ref{twisted-Cartesian}) is the following corollary.
\begin{cor}
Let $A$
%be a simplicial complex or a  and transitive category with
have a twisted structure $\delta$ in a simplicial set $Y$.  Then the twisted homology
 $H_*^{\twist}(A;G)$ (twisted cohomology $H^*_{\twist}(A;G)$) is the homology
 (cohomology) of the chain complex $C^{\twist}_*(A)$ with
coefficients in $G$, where
$$
C_n^{\twist}(A)=\Z(Y_n)\otimes \Z(S(A)_n)
$$
with the differential $\partial^{\twist}_n\colon C_n^{\twist}(K)\to C_{n-1}^{\twist}(K)$
 given by the formula
$$
\partial_n(y\otimes v_0\alpha_1v_1\cdots\alpha_nv_n)=\sum_{i=0}^n(-1)^i\delta_{v_i}
(d_iy)\otimes d_i(v_0\alpha_1v_1\cdots\alpha_nv_n)
$$
for $n\geq0$.\hfill $\Box$
\end{cor}

From the second assertion of Proposition~\ref{twisted-Cartesian},  there is a connection
 between $\delta$-twisted Caretesian products and  simplicial fibre bundles
 which is described as follows. Let $S=S(A)$
 %(or $S(\calC)$) for a simplicial complex $K$ (or a  and transitive category $\calC$)
 be with a nonsingular twisted structure $\delta$ in a simplicial set $Y$. Observe that
  the coordinate projection
\begin{equation}\label{projection-map}
p\colon Y\times_{\delta}S\to S, \quad (g,v_0\alpha_1v_1\cdots\alpha_nv_n))\mapsto v_0\alpha_1v_1\cdots\alpha_nv_n)
\end{equation}
 is a simplicial map. Recall from~\cite[page. 155]{Curtis} that a simplicial map $p\colon E\to B$ is called a \textit{simplicial fibre bundle} with fibre $F$ if for each $b\in B_n$, there is a commutative diagram
 \begin{equation}\label{fibre-bundle}
\begin{diagram}
\Delta[n]\times F&\rTo^{\alpha_b}_{\cong}& E' &\rTo^{\tilde f_b}& E\\
 \dTo>{\mathrm{proj.}}    &  &\dTo&\textrm{pull-back}&\dTo\\
\Delta[n]     &\rEq &\Delta[n]&\rTo^{f_b}&B,\\
\end{diagram}
\end{equation}
where $\Delta[n]$ is the standard simplicial $n$-simplex,  $f_b\colon \Delta[n]\to B$ is the representing map of $b\in B_n$ and $\alpha_b$ is a simplicial isomorphism.

\begin{thm}\label{theorem-fibre-bundle}
Let $S=S(A)$
%(or $S(\calC)$) for a simplicial complex $K$ (or a  and transitive category $\calC$)
 be with a non-singular twisted structure $\delta$ in a simplicial set $Y$.  Then
 the coordinate projection map
$$p\colon Y\times_{\delta}S\to S$$ is a simplicial fibre bundle with fibre $Y$.
\end{thm}
\begin{proof}
Let $\sigma=(0,1,\ldots,n)\in \Delta[n]_n$ be the standard non-degenerate  $n$-simplex.
 For each $b=v_0\alpha_1v_1\cdots\alpha_nv_n\in S_n$,  we have $f_b(\sigma)=b$ by
  definition. By taking iterated faces, we have $f_b(i)=v_i$. The twisted structure
   on $A$
   %or $\calC$
   induces a function
$$
\delta\colon V(\Delta[n])=\{(0),(1),\ldots,(n)\}\longrightarrow \End(Y)
 $$
 given by
$$
\delta_{(i)}=\delta_{v_i}
$$
for $0\leq i\leq n$. For $0\leq i<j\leq n$, since there is an edge between $v_i$ and $v_j$ or an arrow %from
$v_i\to v_j$, we have
$$
\delta_{(i)}\circ \delta_{(j)}=\delta_{(j)}\circ\delta_{(i)}
$$
by the commuting rule
%that
$\delta_{v_i}\circ\delta_{v_j}=\delta_{v_j}\circ\delta_{v_i}$ and so a twisted
structure on $\Delta[n]$ in $Y$. From the definition of $\delta$-twisted Caretesian product, we have the simplicial map
$$
(\id, f_b)\colon Y\times_{\delta}\Delta[n]\longrightarrow E=Y\times_{\delta}S
$$
inducing a fibrewise simplicial isomorphism
\begin{diagram}
Y\times_{\delta}\Delta[n]&\rTo& E'\\
\dTo>{p}&&\dTo\\
\Delta[n]&\rEq& \Delta[n],\\
\end{diagram}
where $E'$ is given by the pull-back in diagram~(\ref{fibre-bundle}) with $E=Y\times_{\delta}S$ and $B=S$. Now we show that there is a fibrewise simplicial isomorphism
\begin{diagram}
Y\times\Delta[n]&\rTo^{\alpha_b}_{\cong}&Y\times_{\delta}\Delta[n]\\
\dTo>{p}&&\dTo>{p}\\
\Delta[n]&\rEq&\Delta[n].\\
\end{diagram}
The simplicial isomorphism $\alpha_b$ is constructed by untwisting the twisted faces and degeneracies. More precisely, observe that, each element $w$ in $\Delta[n]_q$ can be uniquely expressed as a monotone sequence
$$
w=(\overbrace{j_1(w),\ldots, j_1(w)}^{l_1(w)},\ldots,\overbrace{j_{t_w}(w),\ldots,j_{t_w}(w)}^{l_{t_w}(w)})
$$
with $l_i(w)\geq 1$ for $1\leq i\leq t$, $0\leq j_1(w)<j_2(w)<\cdots<j_{t_w}(w)\leq n$  and $\sum\limits_{i=1}^{t_w}l_i(w)=q+1$. Let
$$
\{i_1(w),\ldots, i_{s_w}(w)\}=\{0,1,\ldots,n\}\smallsetminus\{j_1(w),j_2(w),\ldots,j_{t_w}(w)\}
$$
with $i_1(w)<i_2(w)<\cdots<i_{s_w}(w)$ be the set of missing vertices in $w$.
Define the function
$$
\alpha_b\colon Y_q\times \Delta[n]_q \longrightarrow Y_q\times\Delta[n]_q
$$
by setting
$$
\alpha_b(y, w)=(\delta_{(j_{t_w}(w))}^{-l_{t_w(w)}+1}\cdots \delta_{(j_1(w))}^{-l_1(w)+1}\delta_{(i_1(w))}\cdots \delta_{(i_{s_w}(w))}(y), w).
$$
Clearly $
\alpha_b\colon Y_q\times \Delta[n]_q \longrightarrow Y_q\times\Delta[n]_q
$ is bijective for $q\geq0$. It suffices to show that $\alpha_b$ is a simplicial map. We shorten the notation $l_i(w)$ as $l_i$, and similarly for other functions on $w$. Note that, for $0\leq i\leq q$,
$$
\begin{array}{rcl}
d^{\delta}_i(\alpha_b(y,w))&=&d^{\delta}_i(\delta_{(j_t)}^{-l_t+1}\cdots \delta_{(j_1)}^{-l_1+1}\delta_{(i_1)}\cdots \delta_{(i_s)}(y), w)\\
&=&(\delta_{(i')}(d_i\delta_{(j_t)}^{-l_t+1}\cdots \delta_{(j_1)}^{-l_1+1}\delta_{(i_1)}\cdots \delta_{(i_s)}(y)),d_i(w))\\
&=&(\delta_{(i')}\delta_{(j_t)}^{-l_t+1}\cdots \delta_{(j_1)}^{-l_1+1}\delta_{(i_1)}\cdots \delta_{(i_s)}(d_i(y)),d_i(w)),\\
\end{array}
$$
where $(i')$ is the $i$-th vertex of the $q$-simplex $w$. Let $k$ be the minimal integer such that
$$
l_1+\cdots+l_k-1\geq i.
$$
Then $i'=j_k$ and
$$
d^{\delta}_i(\alpha_b(y,w))=(\delta_{(j_t)}^{-l_t+1}\cdots\delta_{j_k}^{-l_k+2}\cdots \delta_{(j_1)}^{-l_1+1}\delta_{(i_1)}\cdots \delta_{(i_s)}(d_i(y)),d_i(w)).
$$
Note that
$$
d_iw=(\overbrace{j_1(w),\ldots, j_1(w)}^{l_1(w)},\ldots,\overbrace{j_{k}(w),\ldots,j_{k}(w)}^{l_k(w)-1},\ldots, \overbrace{j_{t_w}(w),\ldots,j_{t_w}(w)}^{l_{t_w}(w)}).
$$
From the definition of $\alpha_b$, we have
$$
\alpha_b(d_i(y),d_i(w))=d^{\delta}_i(\alpha_b(y,w)).
$$
Consider the degeneracy operation $s_i$, $0\leq i\leq q$. We have
$$
s_iw=(\overbrace{j_1(w),\ldots, j_1(w)}^{l_1(w)},\ldots,\overbrace{j_{k}(w),\ldots,j_{k}(w)}^{l_k(w)+1},\ldots, \overbrace{j_{t_w}(w),\ldots,j_{t_w}(w)}^{l_{t_w}(w)}).
$$
Then
$$
\begin{array}{rcl}
s^{\delta}_i(\alpha_b(y,w))&=&s^{\delta}_i(\delta_{(j_t)}^{-l_t+1}\cdots \delta_{(j_1)}^{-l_1+1}\delta_{(i_1)}\cdots \delta_{(i_s)}(y), w)\\
&=&(\delta_{(j_k)}^{-1}(s_i\delta_{(j_t)}^{-l_t+1}\cdots \delta_{(j_1)}^{-l_1+1}\delta_{(i_1)}\cdots \delta_{(i_s)}(y)),s_i(w))\\
&=&(\delta_{(j_k)}^{-1}\delta_{(j_t)}^{-l_t+1}\cdots \delta_{(j_1)}^{-l_1+1}\delta_{(i_1)}\cdots \delta_{(i_s)}(s_i(y)),s_i(w))\\
&=&(\delta_{(j_t)}^{-l_t+1}\cdots\delta_{(j_k)}^{-l_k}\cdots \delta_{(j_1)}^{-l_1+1}\delta_{(i_1)}\cdots \delta_{(i_s)}(s_i(y)),s_i(w))\\
&=&\alpha_b(s_i(y),s_i(w)).\\
\end{array}
$$
%This finishes the proof.
\end{proof}

\subsection{The $\delta$-Twisted Smash Products}
Now we consider the pointed constructions.
%Let $K$ be a simplicial complex or a category with vertex set $V(K)$.
Let $Y$ be a pointed simplicial set with the base-point $\ast=\{s^n_0\ast\}_{n\geq0}$.
A \textit{reduced twisted structure} on $A$ in $Y$ is a function
$$
\delta\colon V(K)\longrightarrow \End_*(Y), \quad v\mapsto \delta_v
$$
such that the commuting rule~(\ref{commuting-rule2}) holds,  where $\End_*(Y)$ is the
 monoid of base-point-preserving self-simplicial maps of $Y$. A reduced twisted structure
  on $A$ is called \textit{non-singular} if $\delta_v\colon Y\to Y$ is a simplicial
   isomorphism for each $v\in V(A)$.

Suppose that $S=S(A)$
%(or $S(\calC)$) for a simplicial complex $K$
%(or a  and transitive category $\calC$) with
has non-singular reduced twisted structure $\delta$ in a pointed simplicial set $Y$.
Since $\delta_v(\ast)=\ast$ for $v\in V(A)$
%(or $V(\calC)$),
there is a canonical inclusion
$$
\ast\times S=\ast\times_{\delta}S \rInto Y\times_{\delta}S.
$$
Let $a_0$ be a vertex of $A$
%or $\calC$
treated as the base-point with the induced $\delta$-structure by the restriction.
Then there is a canonical simplicial inclusion
$$
Y\times_{\delta} a_0\longrightarrow Y\times_{\delta}S.
$$

The $\delta$-twisted smash product $Y\wedge_{\delta}S$  is then defined as
the simplicial quotient set
\begin{equation}\label{twisted-smash}
Y\wedge_{\delta}S:=(Y\times_{\delta}S)/((Y\times_{\delta}a_0)\cup(\ast\times_{\delta}S)).
\end{equation}

\subsection{The classifying spaces of  the twisted simplicial groups}

Recall that there is a functor $\bar W$ from simplicial groups to simplicial sets, which plays the role of classifying space functor in simplicial theory. The construction of the functor $\bar W$ is briefly reviewed using the terminology of categorical digraphs as follows.

Let $G$ be a group (without assuming simplicial structure). Let $\calG$ be the category with a single vertex $a$ and a collection of arrows labeled by the elements $g\in G$, where the identity  $1=e_a\colon a\to a$  is considered as the trivial path of the vertex $a$. Then $\calG$ is a categorical digraph with the composition operation induced by the multiplication of the group $G$. Let $W(G)=S(\calG)$ defined in section~\ref{section2}.

Now let $G=\{G_n\}_{n\geq0}$ be a simplicial group. The face homomorphisms $d_i\colon G_n\to G_{n-1}$ and the degeneracy homomorphisms $s_i\colon G_n\to G_{n+1}$ induce simplicial face maps $d_i= W(d_i)\colon  W(G_n)\to  W(G_{n-1})$ and simplicial degeneracy maps $s_i= W(s_i)\colon  W(G_n)\to  W(G_{n+1})$ for $0\leq i\leq n$ so that $\{ W(G_{n})\}_{n\geq0}$ is a bi-simplicial set. The simplicial set $\bar W(G)$ is then defined as the diagonal simplicial set of the bi-simplicial set $W(G_*)_*$.

\begin{thm}\label{main-theorem}
Let $S=S(A)$
%for a simplicial complex $K$ or $S(\calC)$ for a  and transitive category $\calC$ with
has a non-singular twisted structure in a simplicial group $G$.
Then there is a natural simplicial map
$$
\theta\colon \bar W(G)\wedge_{\delta} S\longrightarrow \bar W(F^G_{\delta}[S]),
$$
which is a homotopy equivalence after geometric realization, where the twisted structure of $A$
%or $\calC$
in the simplicial set  $\bar W(G)$ is induced from its twisted structure in $G$
through the functor $\bar W$.
\end{thm}
\begin{proof}
Let $v$ be a vertex of $A$. %or $\calC$.
By applying the functor $\bar W$ to the simplicial isomorphism $\delta_v\colon G\to G$,
we have a sequence of isomorphisms of pointed simplicial sets
$$
\bar W(\delta^n_v)\colon \bar W(G_n)\longrightarrow \bar W(G_n)
$$
 and the commutative diagram
\begin{diagram}
\bar W(G_{n-1})&\lTo^{\bar W(d_i)}&\bar W(G_n)&\rTo^{\bar W(s_i)}&\bar W(G_{n+1})\\
\dTo>{\bar W(\delta^{n-1}_v)}&&\dTo>{\bar W(\delta^n_v)}&&\dTo>{\bar W(\delta^{n+1}_v)}\\
\bar W(G_{n-1})&\lTo^{\bar W(d_i)}&\bar W(G_n)&\rTo^{\bar W(s_i)}&\bar W(G_{n+1})\\
\end{diagram}
for $0\leq i\leq n$ and $v\in V(A)$.
%(or $v\in V(\calC)$).

Recall from the definition that the group $(F^G_{\delta}[S])_n=F^G[S]_n$ is the quotient group of the free product
$
\ast_{x\in S_n} (G_n)_x
$
subject to the relations
%that
$(G_n)_{a_0^{n}}=\{1\}$. For $x\in S_n$, the inclusion $(G_n)_x\rInto F^G[S]_n$
induces a simplicial map
$$
j_x\colon \bar W((G_n)_x)\longrightarrow \bar W((F^G_{\delta}[S])_n)
$$
and so a simplicial map
$$
\tilde\theta_n\colon \bar W(G_n)\times S_n\longrightarrow \bar W((F^G_{\delta}
[S])_n)\quad (y,x)\mapsto j_x(y),
$$
where we consider $S_n$ as a discrete simplicial set with $(S_n)_q=S_n$ for $q\geq0$
and faces and degeneracies being given by the identity map. Since
 $(G_n)_{a_0^{n}}=\{1\}$, we have
$$
\tilde\theta_n(\bar W(G_n)\times\{a_0^n\})=\ast.
$$
Since $j_x(\ast)=\ast$ for any $x\in S_n$, we have
$$
\tilde\theta_n(\ast\times S_n)=\ast
$$
and so $\tilde\theta$ factors through the smash product. Let
$$
\theta_n\colon \bar W(G_n)\wedge S_n\longrightarrow \bar W((F^G_{\delta}[S])_n)
$$
be the resulting simplicial map. Now we let $n$ be varied. Let
$$
x=v_0\alpha_1v_1\cdots\alpha_nv_n\in S_n=(S_n)_q\textrm{ and }
$$
$$
y=ag_1ag_2a\cdots g_qa\in \bar W(G_n)_q
$$
with $g_1,\ldots,g_q\in G_n$.
Then
$$
\theta_n(y\wedge x)=a(g_1)_xa(g_2)_xa\cdots  (g_q)_xa.
$$
Under the twisted faces in the simplicial group $F^G_{\delta}[S]$, we have
$$
\begin{array}{rcl}
\bar W(d^{\delta}_i)(\theta_n(y\wedge x))&=&\bar W(d^{\delta}_i)
(a(g_1)_xa(g_2)_xa\cdots (g_q)_xa\\
&=&a(d^{\delta}_i((g_1)_x))a(d^{\delta}_i((g_2)_x))a\cdots  (d^{\delta}_i((g_q)_x))a\\
&=&a((\delta^{n-1}_{v_i}(d_i(g_1)))_{d_ix})a\cdots  ((\delta^{n-1}_{v_i}
(d_i(g_q)))_{d_ix})a\\
&=&\theta_{n-1}((\bar W(\delta^{n-1}_{v_i})(d_i(y)))\wedge d_i(x))\\
\end{array}
$$
for $0\leq i\leq n$.
Similarly, under the twisted faces in the simplicial group $F^G_{\delta}[S]$, we have
$$
\bar W(s_i^{\delta}(\theta_n(y\wedge x))=
\theta_{n+1}((\bar W((\delta_{v_i}^{n+1})^{-1})(s_i(y)))\wedge s_i(x))).
$$
We define a $\delta$-twisted structure on the sequence of simplicial sets
$$\{\bar W(G_n)\wedge S_n\}_{n\geq 0}$$
by the formulae
$$
d^{\delta}_i(y\wedge x)=(\bar W(\delta^{n-1}_{v_i})(d_i(y)))\wedge d_i(x),
$$
$$
s^{\delta}_i(y\wedge x)=(\bar W((\delta_{v_i}^{n+1})^{-1})(s_i(y)))\wedge s_i(x))
$$
for $0\leq i\leq n$, $x=v_0\alpha_1v_1\cdots\alpha_nv_n\in S_n=(S_n)_q$ and $y\in \bar W(G_n)_q$. Since each $\theta_n$ is injective and $\bar W((F^G_{\delta}[S])_*)_*$ is a bi-simplicial set, the simplicial identities hold for $d^{\delta}_i$ and $s^{\delta}_i$ in $T_{*,*}=\{\bar W(G_n)\wedge S_n\}_{n\geq 0}$ so that $T_{\ast,\ast}$ is a bi-simplicial set with
$$
\theta_*\colon T_{\ast,\ast}\longrightarrow \bar W((F^G_{\delta}[S])_*)_*
$$
being a morphism of bi-simplicial sets.

Recall that, by the Whitehead Theorem~\cite[Proposition 4.3]{Kan-Thurston}, the canonical inclusion $$\bar W(G')\vee \bar W(G'')\to \bar W(G'\ast G'')$$ is a homotopy equivalence after geometric realization for any simplicial groups $G'$ and $G''$. It follows that
$$
\theta_n\colon \bar W(G_n)\wedge S_n\longrightarrow \bar W((F^G_{\delta}[S])_n)
$$
is a homotopy equivalence after geometric realization for $n\geq 0$. According to~\cite{Bousfield}, the geometric realization
$$
|\theta_*|\colon |T_{*,*}|\longrightarrow | \bar W((F^G_{\delta}[S])_*)_*|
$$
is a homotopy equivalence.  Moreover, by~\cite{Bousfield}, the geometric realization
of the diagonal simplicial simplicial associated to a bi-simplicial set $X_{\ast,\ast}$
 is homotopy equivalent to $|X_{\ast,\ast}|$. From the definition of the functor
 $\bar W$, the simplicial set $\bar W(F^G_{\delta}[S])$ is given by the diagonal
  simplicial set of the bi-simplicial set $\bar W((F^G_{\delta}[S])_*)_*$. By the
   definition of $\delta$-twisted smash product, it is evident that the diagonal
    simplicial set of $T_{\ast,\ast}$ is $\bar W(G)\wedge_{\delta}S$.
    %The proof is finished.
\end{proof}

\section{Twisted homology of simplicial complexes and  categories}\label{section5}

In this section, we give some remarks on the twisted $\Delta$-groups $F^{G,\Delta}_{\delta}[S]$, where $S=S(A)$
%or $S(\calC)$ for a simplicial complex $A$ or a  category $\calC$
with a twisted structure $\delta$ in a simplicial group $G$. The construction
$F^{G,\Delta}_{\delta}[S]$ seems to be general, which only requires the commuting
 rule~(\ref{equation1.1}), and so one has the homology groups from the chain complexes
  given by $\Z(F^{G,\Delta}_{\delta}[S])$ and $(F^{G,\Delta}_{\delta}[S])^{\ab}$
  (as well as $\Z((\bar W(G)\times_{\delta,\Delta}S)$ in section~\ref{subsection4.1})
  for any twisted structure. On the other hand, the $\Delta$-groups
 $F^{G,\Delta}_{\delta}[S]$ could be very wild in general. For instance, if $\delta_v$
is   the trivial endomorphism of $G$ for any vertex $v$, all twisted face operations
in $F^{G,\Delta}_{\delta}[S]$ become trivial, which concludes that the differentials
in $\Z(F^{G,\Delta}_{\delta}[S])$, $(F^{G,\Delta}_{\delta}[S])^{\ab}$ and
$\Z((\bar W(G)\times_{\delta,\Delta}S)$ are zero maps and so their homology are given by
the chains themselves. This indicates that the homology groups from the chain
complexes given by $\Z(F^{G,\Delta}_{\delta}[S])$, $(F^{G,\Delta}_{\delta}[S])^{\ab}$
and $\Z((\bar W(G)\times_{\delta,\Delta}S)$ may not be homotopy invariants for general
 twisted structures $\delta$. If the twisted structure is nonsingular, the homology
  groups from the chain complexes given by $\Z(F^{G,\Delta}_{\delta}[S])$,
  $(F^{G,\Delta}_{\delta}[S])^{\ab}$ and $\Z((\bar W(G)\times_{\delta,\Delta}S)$
  coincide with  the homology groups from the chain complexes given by the simplicial
   abelian groups $\Z(F^{G}_{\delta}[S])$, $(F^{G}_{\delta}[S])^{\ab}$ and
   $\Z((\bar W(G)\times_{\delta}S)$ by ~\cite[Section 5]{Curtis} and so
    Theorems~\ref{theorem-fibre-bundle} and~\ref{main-theorem} assure that these
     homology groups only depend on the homotopy type of the corresponding fibre bundles.

Some simplicial techniques can be used for understanding the twisted $\Delta$-groups $F^{G,\Delta}_{\delta}[S]$.

\begin{prop}
Let $S=S(K)$ for a simplicial complex $K$ with a twisted structure in a simplicial group $G$. Let $K_1$ and $K_2$ be simplicial
sub-complexes of $K$ such that $K_1\cup K_2=K$ and $K_1\cap K_2\not=\emptyset$.
Let $\delta_{K_i}$ be the restriction of the twisted structure $\delta_K$ on $K_i$.
We choose the basepoint in $K_1\cap K_2$. Then
\begin{enumerate}
\item The $\Delta$-group
$$
F^{G, \Delta}_{\delta}[S(K)]=F^{G,\Delta}_{\delta}[S(K_1)]
\ast_{F^{G,\Delta}_{\delta}[S(K_1\cap K_2)]}F^{G,\Delta}_{\delta}[S(K_2)].
$$
is the free product with amalgamation.
\item Suppose that $G$ is a simplicial abelian group. Then there is a short exact sequence of the chain complexes
$$
F^{G,\Delta}_{\delta}[S(K_1\cap K_2)]^{\ab}\rInto F^{G,\Delta}_{\delta}
[S(K_1)]^{\ab}\oplus F^{G,\Delta}_{\delta}[S(K_1)]^{\ab} \rOnto
F^{G,\Delta}_{\delta}[S(K)]^{ab}
$$
\end{enumerate}
\end{prop}
\begin{proof}
Assertion (2) follows from assertion (1) by taking abelianization. For proving
assertion (1), we check that $S(K)=S(K_1)\cup_{S(K_1\cap K_2)}S(K_2)$. By the
definition, the elements in $S(K)_n$ are given by $(v_0,\ldots,v_n)$ with
$v_0\leq v_1\leq\cdots\leq v_n$ such that $\{v_0,\ldots,v_n\}$ form a simplex in $K$.
Since $K=K_1\cup K_2$, the simplex $\{v_0,\ldots,v_n\}$ is either in $K_1$ or $K_2$.
This shows that $S(K)=S(K_1)\cup_{S(K_1\cap K_2)}S(K_2)$. Assertion (1) then follows
from %the definition of
our construction.
\end{proof}
The second assertion in the above proposition assures that the Mayer-Vietoris
sequence can be applied in the homology theory given by $H_* (F^{G,\Delta}_{\delta}
[S(K)])^{ab}$ for twisted simplicial complexes. Similar results hold for the case
$S(\calC)$ for   category $\calC$ with twisted structure.

Now let us consider the special case of the twisted $\Delta$-groups of the cones of
simplicial complexes and categories. For a simplicial complex $K$, the cone
$CK=a\ast K$ is defined in the usual way. For a category $\calC$, the cone
$C\calC=a\ast\calC$ is the category obtained by adding an initial vertex $a$ with
assigning a unique arrow $e_{a,v}\colon a\to v$ for each $v\in V(\calC)$ and the
identical arrow $\id_a:a\to a$. Suppose that $\delta\colon V(K)\to \End(G)$ or
$V(\calC)\to \End(G)$ is a twisted structure in a simplicial group $G$. A twisted
structure $\delta_{CK}$ (or $\delta_{C\calC}$) is called a \textit{regular extension} of
the twisted structure $\delta_{K}$ (or $\delta_{\calC}$) if
\begin{enumerate}
\item $\delta_{CK}|_{V(K)}=\delta_{K}$ (or $\delta_{C\calC}|_{V(\calC)}=\delta_{\calC}$)
and
\item $\delta(a)\in \Aut(G)$.
\end{enumerate}
We recall that,  since there is an edge or arrow between $a$ and any other vertex, the commuting rule~(\ref{equation1.1}) of the $\delta$-structure forces that $\delta_a\circ\delta_v=\delta_v\circ\delta_a$ for any vertex $v$.

Recall also that any group $G$ can be considered as a discrete simplicial group with
 $G_n=G$ and faces and degeneracies being the identity map.
\begin{prop}
Let $S=S(a\ast A)$ %(or $S(a\ast\calC)$)  for
where $A$ is a simplicial complex $K$ or  a   category $\calC$ with a twisted structure given as a regular extension of a twisted structure on $A$
%or $\calC$
in a group $G$. Then the following chain complexes
$$
\Z(F^{G,\Delta}_{\delta}[S]), \quad F^{G,\Delta}_{\delta}[S]^{\ab}$$
are contractible.
\end{prop}
\begin{proof}
Let
$$
\Phi\colon (F^{G,\Delta}_{\delta}[S])_n=F^{G}[S]_n\longrightarrow
(F^{G,\Delta}_{\delta}[S])_{n+1}=F^{G}[S]_{n+1}
$$
be the (unique) group homomorphism such that
$$
\Phi|_{(G_n)_{v_0\alpha_1v_1\cdots \alpha_nv_n}}(g_{v_0\alpha_1v_1\cdots \alpha_nv_n})=
(\delta_a)^{-1}(g)_{ae_{a,v_0}v_0\alpha_1v_1\cdots \alpha_nv_n}.
$$
Let us compute $d^{\delta}_i\circ \Phi$. We shorten the notation $d^{\delta}_i$ as $d_i$.
$$
\begin{array}{rcl}
d_0 \Phi(g_{v_0\alpha_1v_1\cdots \alpha_nv_n})&=&d_0((\delta_a)^{-1}
(g)_{ae_{a,v_0}v_0\alpha_1v_1\cdots \alpha_nv_n})\\
%&=&d_0((\delta_a)^{-1}(g))_{ae_{a,v_0}v_0\alpha_1v_1\cdots \alpha_nv_n})\\
&=&(\delta_a(d_0((\delta_a)^{-1}(g))))_{v_0\alpha_1v_1\cdots \alpha_nv_n}\\
&=&(\delta_a((\delta_a)^{-1}(g))_{v_0\alpha_1v_1\cdots \alpha_nv_n}
\textrm{  because $G$ is discrete}\\
&=&g_{v_0\alpha_1v_1\cdots \alpha_nv_n}.\\
\end{array}
$$
Thus
\begin{equation}\label{equation5.1}
d_0\Phi=\mathrm{id}.
\end{equation}
Now let $i>0$. Then
$$
\begin{array}{rcl}
d_i\Phi(g_{v_0\alpha_1v_1\cdots\alpha_nv_n})&=&d_i((\delta_a)^{-1}
(g)_{ae_{a,v_0}v_0\alpha_1v_1\cdots\alpha_nv_n})\\
&=&(\delta_{v_{i-1}}(d_i(\delta_a)^{-1}(g)))_{ae_{a,v_0}v_0\alpha_1v_1\cdots
v_{i-2}\alpha_{i-1}\alpha_iv_i \cdots\alpha_nv_n}\\
%&=&(d_is_0(\delta_{v_{i-1}}((\delta_a)^{-1}(g))))_{ae_{a,v_0}v_0\alpha_1v_1\cdots
%v_{i-2}\alpha_{i-1}\alpha_iv_i \cdots\alpha_nv_n}\\
%&=&(s_0 d_{i-1}((\delta_a)^{-1}(\delta_{v_{i-1}}(g))))_{ae_{a,v_0}v_0\alpha_1v_1\cdots
%v_{i-2}\alpha_{i-1}\alpha_iv_i \cdots\alpha_nv_n}\\
&=&((\delta_a)^{-1}\delta_{v_{i-1}}g)_{ae_{a,v_0}v_0\alpha_1v_1\cdots
v_{i-2}\alpha_{i-1}\alpha_iv_i \cdots\alpha_nv_n}\\
&&\qquad\textrm{because $G$ is discrete}\\
&=&\Phi d_{i-1}(g_{v_0\alpha_1v_1\cdots\alpha_nv_n}).\\
\end{array}
$$
Thus we have the following identity:
\begin{equation}\label{equation5.2}
d_i\Phi=\Phi d_{i-1}.
\end{equation}
It follows that, in the chain complex $\Z(F^{G,\Delta}_{\delta}[S])$,
$$
\begin{array}{rcl}
\partial_{n+1}\circ\Phi&=&\sum\limits_{i=0}^{n+1}(-1)^id_i\Phi\\
&=&\id-\sum\limits_{i=1}^{n+1}(-1)^{i-1}\Phi d_{i-1}\\
&=&\id-\Phi\circ \partial_n.\\
\end{array}
$$
and so $\Z(F^{G,\Delta}_{\delta}[S])$ is contractible. By taking the abelianization, the homomorphisms $\Phi^{\ab}$ defines a null homotopy for the chain complex $F^{G,\Delta}[S]^{\ab}$. %This finishes the proof.

\end{proof}

\end{document}